\documentclass[11pt]{amsart}
\usepackage{amsfonts}
\usepackage{amssymb}
\usepackage{amsmath,amsthm}
\usepackage{amscd}

\newcommand{\beq}{\begin{equation}}
\newcommand{\eeq}{\end{equation}}
\newcommand{\beqa}{\begin{eqnarray}}
\newcommand{\eeqa}{\end{eqnarray}}
\newcommand{\nn}{\nonumber}
\newcommand{\noi}{\noindent}

\newtheorem{definition}{Definition}
\newtheorem{proposition}{Proposition}
\newtheorem{theorem}{Theorem}
\newtheorem{corollary}{Corollary}
\newtheorem{lemma}{Lemma}
\theoremstyle{definition}

\begin{document}
\author{Piergiulio Tempesta}
\address{Departamento de F\'{\i}sica Te\'{o}rica II (M\'{e}todos Matem\'{a}ticos de la f\'isica), Facultad de F\'{\i}sicas, Universidad
Complutense de Madrid, 28040 -- Madrid, Spain \\ and
Instituto de Ciencias Matem\'aticas, C/ Nicol\'as Cabrera, No 13--15, 28049 Madrid, Spain}
\email{p.tempesta@fis.ucm.es, piergiulio.tempesta@icmat.es}
\title[formal groups, AM congruences and zeta functions]{The Lazard formal group, universal congruences and special values of zeta functions}
\date{July 7, 2015}
\maketitle

\begin{abstract} A connection between the theory of formal groups and arithmetic number theory is established. In particular, it is shown how to construct general Almkvist--Meurman--type congruences for the universal Bernoulli polynomials that are related with the Lazard universal formal group \cite{Tempesta1}-\cite{Tempesta3}. Their role in the theory of $L$--genera for multiplicative sequences is illustrated. As an application, sequences of integer numbers are constructed. New congruences are also obtained, useful to compute special values of a new class of Riemann--Hurwitz--type zeta functions.
\end{abstract}

\tableofcontents


\section{Introduction: Formal group laws}

The theory of formal groups \cite{Boch}, \cite{Haze} has been intensively investigated in the last decades, due to its relevance in many branches of mathematics, especially algebraic topology \cite{BMN},  \cite{BCRS}, \cite{Quillen}, \cite{Faltings}, the theory of elliptic curves \cite{Serre2}, and arithmetic number theory \cite{Adel}, \cite{Adel2}, \cite{Tempesta1}, \cite{Tempesta3}.


Given a commutative ring $R$ with identity, and the ring $R\left\{ x_{1},\text{ }%
x_{2},..\right\} $ of formal power series in the variables $x_{1}$, $x_{2}$,
... with coefficients in $R$, a commutative one--dimensional formal group law
over $R$ is a formal power series $\Phi \left( x,y\right) \in R\left\{
x,y\right\} $ such that \cite{Boch}
\begin{equation*}
1)\qquad \Phi \left( x,0\right) =\Phi \left( 0,x\right) =x,
\end{equation*}%
\begin{equation*}
2)\qquad \Phi \left( \Phi \left( x,y\right) ,z\right) =\Phi \left( x,\Phi
\left( y,z\right) \right) \text{.}
\end{equation*}
When $\Phi \left( x,y\right) =\Phi \left( y,x\right) $, the formal group law is
said to be commutative (the existence of an inverse formal series $\varphi
\left( x\right) $ $\in R\left\{ x\right\} $ such that $\Phi \left( x,\varphi
\left( x\right) \right) =0$ follows from the previous definition).

The relevance of formal group laws relies first of all on their close connection with group theory. Precisely, a formal group law $\Phi(x,y)$ defines a functor $F: \bf{Alg}$$_{R} \longrightarrow \bf{Group}$, where $\textbf{Alg}$$_R$ denotes the category of commutative unitary algebras over $R$ and $\textbf{Group}$ denotes the category of groups \cite{Haze}. The functor $F$ is by definition the formal group (sometimes called the formal group scheme) associated to the formal group law $\Phi(x,y)$.

As is well known, over a field of characteristic zero, there exists an equivalence of categories between Lie algebras and formal groups.

%
%
%

Given a formal group law $\Phi(x,y)$, let $\Phi_2(x,y)$ denote its quadratic part. The Lie algebra (over the same ring) associated to $\Phi(x,y)$ is defined via the identification $[x,y]=\Phi_2(x,y)-\Phi_2(y,x)$. This equivalence of categories is no longer true in a field of characteristic $p\neq0$.

A crucial point is the relation between formal groups and cobordism theory \cite{BMN}, \cite{Nov}, \cite{Quillen}, \cite{Ray}.
Given a topological space $S$, consider the complex cobordism theory of $S$, i.e. the generalized cohomology theory $MU^{*}(S)$ defined by the Thom spectrum $MU$.

Then there exists  a universal commutative formal group law $F_{MU}$ associated to the cobordism theory. It turns out that its logarithm is $log_{F_MU} (x) =\sum_{k\geq 0} [\mathbb{C}P^{k}]x^{k+1}/k +1$, where $[\mathbb{C}P^{k}]$ is the cobordism class of the complex projective space $\mathbb{C}P^{k}$ of complex dimension $k$.

A nice connection with combinatorial Hopf algebras and Rota's umbral calculus \cite{Rota} was found in \cite{BCRS}. A combinatorial approach has also been proposed in \cite{AB}.

In the papers \cite{Tempesta1}, \cite{Tempesta3} a relation between the Lazard formal group and the theory of L--series has been established. A new class of Bernoulli--type polynomials, called the \textit{Universal Bernoulli polynomials}, have been introduced, which were preliminarily studied in \cite{Tempesta2}. They are Appell polynomials defined via the universal group exponential law.

In \cite{MT1}, \cite{MT2}, the universal Bernoulli polynomials have been related to the theory of hyperfunctions of one variable by means of an extension of the classical Lipschitz summation formula to negative powers.

Due to the generality of their definition, the universal polynomials include many (if not all) of the generalizations of the classical Bernoulli polynomials known in the literature. In \cite{Adel}, \cite{Adel2}, Kummer--type congruences for the universal Bernoulli numbers have been derived.

Another interesting aspect is represented by the relation among Bernoulli numbers, Hirzebruch's theory of genera, and the computation of Todd classes \cite{Hirz}. Indeed, in the context of complex cobordism theory, let  $\varphi$  be a multiplicative genus in the sense of Hirzebruch, and  $c_n$  be the value of  $\varphi$  on the complex projective $n$--space. Then the generating function $t/G(t)$, where $G(t)$ is the universal formal group exponential (see Definition \ref{def1}),  is the corresponding characteristic power series in Hirzebruch's formalism. For instance, the power series associated to the $L$--genus, the Todd genus and the $A$--genus are special cases of the proposed construction. Consequently, the generating function defining the polynomials $B_{k,a}^G(x)$ can be interpreted as some special value of the genus.

In this paper, we address the following general problem: \textit{Construct congruences for the Universal Bernoulli polynomials and apply them to compute special values of the related zeta functions.}

\textbf{Main result}. A universal congruence, that generalizes the well--known Almkwist--Meurman \cite{AM} and Bartz--Rutkowski \cite{BR} congruences to the case of the universal Bernoulli polynomials, is proposed.
Two families of polynomials, related to the characteristic power series of important $m$--sequences of the genera theory of algebraic topology, are introduced. They represent nontrivial representations of the universal polynomials. Their arithmetic properties are studied.

A second result, coming from the theory of universal congruences, is a procedure enabling one to construct infinitely many sequences of integers and polynomials with integral coefficients.

Finally, we deduce congruences for the values at negative integers of a new class of Hurwitz--type zeta functions, that have been recently introduced in \cite{Tempesta3}.
Work is in progress on a $p$--adic generalization of the previous theory.

\section{The Lazard universal formal group, the universal Bernoulli polynomials and numbers}

We recall some basic definitions, necessary in the subsequent discussion.

\begin{definition}\label{def1}
Consider the  ring $B=\mathbb{Z}[c_{1},c_{2},...]$ of integral polynomials in infinitely many variables, and the formal group logarithm
\begin{equation}
F\left( s\right) = \sum_{i=0}^{\infty} c_i \frac{s^{i+1}}{i+1},
\label{I.1}
\end{equation}%

\noindent with $c_0=1$. Let $G\left( t\right)$ be its compositional inverse (the formal group exponential):
\begin{equation}
G\left( t\right) =\sum_{i=0}^{\infty} \gamma_i \frac{t^{i+1}}{i+1} \label{I.2}
\end{equation}%
so that $F\left( G\left( t\right) \right) =t$. We have $\gamma_{0}=1, \gamma_{1}=-c_1, \gamma_2= \frac{3}{2} c_1^2 -c_2,\ldots$.
The Lazard formal group law \cite{Haze} is defined by the formal power series
\begin{equation}
\Phi \left( s_{1},s_{2}\right) =G\left( F\left(
s_{1}\right) +F\left( s_{2}\right) \right).
\end{equation}

\end{definition}
The coefficients of the power series $G\left( F\left( s_{1}\right) +F\left(
s_{2}\right) \right)$ lie in the ring $B \otimes \mathbb{Q}$ and generate over $\mathbb{Z}$ a subring $A \subset B \otimes \mathbb{Q}$, called the Lazard ring $L$.

For any commutative one-dimensional formal group law over any ring $R$, there exists a unique homomorphism $L\to R$ under which the Lazard group law is mapped into the given group law (the so called \textit{universal property} of the Lazard group).

In \cite{Tempesta1}, the following family of polynomials has been introduced.

\begin{definition}
The universal higher--order Bernoulli polynomials $\widehat{B}_{k}^{(\alpha)}\left(
x,c_{1},...,c_{n}\right) \equiv \widehat{B}_{k}^{(\alpha)}\left( x\right) $ are defined by the relation

\begin{equation}
\left( \frac{t}{G\left( t\right) }\right) ^{\alpha}e^{xt}=\sum_{k\geq
0}\widehat{B}_{k}^{(\alpha)}\left( x\right) \frac{t^{k}}{k!}\text{,}\qquad \qquad x,\alpha\in
\mathbb{R}.  \label{I.3}
\end{equation}
\end{definition}

\noi (wherever the l.h.s. makes sense). Notice that if $\alpha=1$, $c_{i}=\left( -1\right) ^{i}$, then $F\left( s\right) =\log
\left( 1+s\right) \,$, $G\left( t\right) =e^{t}-1$, and the universal
Bernoulli polynomials and numbers reduce to the standard ones. For sake of
simplicity, we will put $\widehat{B}_{k}^{(1)}\left( x\right) \equiv \widehat{B}_{k}\left(
x\right) $ and $\widehat{B}_{k}^{(1)}\left( 0\right) \equiv \widehat{B}_{k}$.

The quantities $\widehat{B}_{k} \in \mathbb{Q}\left[ c_{1},c_{2},\ldots \right]$ coincide with Clarke's \textit{universal Bernoulli numbers} \cite%
{Clarke}.

The universal Bernoulli numbers satisfy some very interesting congruences. In particular, we will discuss briefly the universal Von Staudt's congruence \cite{Clarke} (that generalizes the classical Clausen--von Staudt congruence and several variations of it known in the literature) and the universal Kummer congruence \cite{Adel}. Both congruences have a distinguished role in algebraic geometry \cite{BCRS}, \cite{Katz}.

The universal Bernoulli polynomials (\ref{I.3}) possess many remarkable
properties. By construction, for any choice of the sequence $\left\{
c_{n}\right\} _{n\in \mathbb{N}}$ they represent a class of \textit{Appell
polynomials}. It means that, for all $k\in \mathbb{N}/\{0\}$, $D \widehat{B}_{k}^{(\alpha)}\left( x\right)= k \widehat{B}_{k-1}^{(\alpha)}\left( x\right)$, where $D$ denotes the derivative operator. The binomial property holds:
\begin{equation}
\widehat{B}_{n}^{(\alpha)}(x+y)=\sum_{m=0}^{n}\binom{n}{m}\widehat{B}^{(\alpha)}_m(x)y^{n-m}\label{2.1}.
\end{equation}

We also mention that in the paper \cite{Tempesta2}, several families of polynomial obtained by specializing eq. \eqref{I.3}, i.e. the
Bernoulli--type polynomials of first and second kind, as well as related
Euler polynomial sequences, were studied.

\section{Generalized Almkvist--Meurman congruences}
\subsection{Some universal congruences}

We mention here only two of the
most relevant properties of the universal Bernoulli numbers $\widehat{B}_{n}$ \cite{Adel}.

\textit{i) The universal Von Staudt's congruence} \cite{Clarke}.

\noindent Let $\widehat{B}_{0}=1$, and if $n>0$ is even, then
\begin{equation}
\widehat{B_{n}}\equiv -\sum_{\overset{p-1\mid n}{p\hspace{1mm}\text{prime}}}\frac{%
c_{p-1}^{n/(p-1)}}{p}\qquad mod\quad \mathbb{Z}\left[ c_{1},c_{2},...\right]
\text{;}  \label{US1}
\end{equation}%
\noindent Let $\widehat{B}_{1}=c_{1}/2$ and if $n>1$ is odd, then
\begin{equation}
\widehat{B_{n}}\equiv \frac{c_{1}^{n}+c_{1}^{n-3}c_{3}}{2}\qquad \hspace{5mm} mod\quad
\mathbb{Z}\left[ c_{1},c_{2},...\right] \text{.}  \label{US2}
\end{equation}

When $c_{n}=(-1)^n$, the celebrated Clausen--Von Staudt congruence for
Bernoulli numbers is obtained.

\smallskip \textit{ii) The universal Kummer congruences} \cite{Adel}, \cite{Adel2}.

The numerators of the classical Bernoulli numbers play a special role, due to the
Kummer congruences and to the notion of regular prime numbers, introduced in
connection with Fermat's  Last  Theorem. The
relevance of Kummer congruences in algebraic geometry has been enlightened
in \cite{BCRS}.
As shown by Adelberg, the numbers $\widehat{B}_{n}$ satisfy a universal congruence. Suppose that $n\neq 0,1$ (mod $p-1$). Then
\begin{equation}
\frac{\widehat{B}_{n+p-1}}{n+p-1}\equiv \frac{\widehat{B}_{n}}{n}%
c_{p-1}\qquad mod\quad p\mathbb{Z}_{p}\left[ c_{1},c_{2},...\right] \text{.}
\end{equation}

We recall that in \cite{Carlitz1}, \cite{Carlitz2}, Carlitz studied Kummer congruences for specific values of the variables $c_i$.

Interesting congruences modulo prime numbers and for the cases $n\equiv 0,1$ mod $p-1$ have been proved by Adelberg in \cite{Adel3} and by Hong, Zhao and Zhao in \cite{HZZ}.

\smallskip Almkvist and Meurman in \cite{AM} discovered a remarkable congruence for the classical Bernoulli polynomials:
\begin{equation}
k^{n}\tilde{B}_{n}\left( \frac{h}{k}\right) \in \mathbb{Z}\text{,}  \label{II.29}
\end{equation}
where $k, h$ are positive integers, $\tilde{B}_n(t):=B_n(t)-B_n(0)$, and $B_n(x)$ are the Bernoulli polynomials. Bartz and Rutkowski \cite{BR} have proved a
theorem which combines the results of Almkvist and Meurman \cite{AM} and
Clausen--von Staudt. Simpler proofs of the original one proposed in \cite{AM} have also been provided in \cite{CS} and \cite{Sury}.

The main result of this section is a generalization of the Almkvist--Meurman (AM) and Bartz--Rutkowski theorems for the universal polynomials (\ref{I.3}). We will adopt the notation $B_{n}^{G}(x)$ to denote the family of generalized Bernoulli polynomials \eqref{I.3} associated with a given $G(t)$ (corresponding to a choice of the indeterminates $c_i$) with $\alpha=1$. Also, $B_{n}^{G}(0):=B_{n}^{G}$.

\begin{theorem}
\label{th1} Let $h> 0,$ $k>0,$ $n\geq 0$ be integers. Consider the polynomials
defined by
\begin{equation*}
\frac{t}{G\left( t\right) }e^{xt}=\sum_{i\geq 0}B_{i}^{G}\left( x\right)
\frac{t^{i}}{i!}\text{,}
\end{equation*}%
where $G\left( t\right)$ is a formal group exponential, such that $c_i \in \mathbb{Z}$ for all $i$. Assume that $c_{p-1}\equiv 0, 1$ mod $p$ for all odd primes $p$, and either $c_1\equiv c_3$ mod $2$, or $c_1$ is odd and $c_3$ even. Then
\begin{equation}
k^{n}\widetilde{B}_{n}^{G}\left( \frac{h}{k}\right) \in \mathbb{Z}\text{,}  \label{II.30}
\end{equation}%
where $\widetilde{B}_{n}^{G}\left( x\right) $ $=$ $B_{n}^{G}\left( x\right) -%
{B}_{n}^{G}$.
\end{theorem}

\begin{proof}
First, we shall establish a result crucial in the proof of Theorem \ref{th1}.


%



\begin{lemma}
\label{lem4} Under the assumptions of Theorem \ref{th1}, we have that
\begin{equation}
\sum_{m=0}^{n-1} \binom{n}{m} k^{m} {B}_{m}^{G} \in \mathbb{Z}. \label{fundcong}
\end{equation}
\end{lemma}
\noi A relation useful in the proof of this statement is the following congruence, valid for any prime number $p\geq2$, first established by Hermite for $n$ odd in \cite{Herm}, and then proved in full generality in \cite{Bach}:
\begin{equation}
\sum_{\overset{m=1}{p-1\mid m}}^{n-1} \binom{n}{m} \equiv 0 \hspace{3mm} mod \hspace{1mm} p. \label{HB}
\end{equation}

\noindent Then, we have
\begin{eqnarray} \label{C3a}
\sum_{\overset{m=0}{m\hspace{1mm}\text{even}}}^{n-1} \binom{n}{m} k^{m} {B}_{m}^{G} \equiv &-&\sum_{\overset{m=2}{m\hspace{1mm}\text{even}}}^{n-1}\sum_{\overset{p\hspace{1mm}\text{odd prime}}{p-1 \mid m}}\binom{n}{m} k^{m}\frac{c_{p-1}^{m/(p-1)}}{p} \hspace{3mm} \\ \nn
&+& \sum_{\overset{m=2}{m\hspace{1mm}\text{even}}}^{n-1}\binom{n}{m}k^m \frac{c_{1}^{m}}{2} \hspace{34mm} mod \hspace{1mm} \mathbb{Z},
\end{eqnarray}
\begin{eqnarray} \label{C3b}
 \sum_{\overset{m=1}{m\hspace{1mm}\text{odd}}}^{n-1} \binom{n}{m} k^{m} {B}_{m}^{G} &\equiv& \frac{n c_1 k}{2}+ \sum_{\overset{m=3}{m\hspace{1mm}\text{odd}}}^{n-1} \frac{c_1^{m}+c_{1}^{m-3}c_3}{2}\binom{n}{m}k^{m}\qquad mod \hspace{1mm} \mathbb{Z}.
\end{eqnarray}
The congruences \eqref{C3a} and \eqref{C3b} hold due to the universal congruences \eqref{US1} and \eqref{US2}, and by taking into account that $B_{0}^{G}=1$.

Let us assume that $c_{p-1} \equiv 1$ mod $p$ for $p$ odd (the case $c_{p-1} \equiv 0$ mod $p$ will follow immediately).

Concerning the congruence \eqref{C3a}, in the r.h.s. we can suppose that in each addend of the first sum  $p\nmid k$; otherwise the addend would contribute an integer. Since $p\nmid k$ and $p-1 \mid m$, we can use Fermat's Little Theorem, and deduce that $k^{m} \equiv 1$ mod $p$. Then the first sum is an integer, due to congruence   \eqref{HB}.

By summing up the congruences  \eqref{C3a} and \eqref{C3b}, we get
\[
\sum_{m=0}^{n-1} \binom{n}{m} k^{m} {B}_{m}^{G} \equiv \sum_{m=1}^{n-1} \binom{n}{m} k^m \frac{c_1^m}{2}+ \sum_{\overset{m=3}{m\hspace{1mm}\text{odd}}}^{n-1} \frac{c_{1}^{m-3}c_3}{2}\binom{n}{m}k^{m} \qquad mod \hspace{1mm} \mathbb{Z}
\]
The first sum in the last congruence is an integer, due to relation \eqref{HB} for $p=2$ and the obvious fact that either the product $k^m c_1^m$ is even, or when is odd, its contribution in each summand does not change the parity of the sum. The remaining sum gives also an integer. Indeed,
\[
\sum_{m=3, \hspace{1mm}\text{m odd}}^{n-1} \binom{n}{m}=\begin{cases} 2^{n-1}-n \qquad\hspace{7mm}\text{if} \hspace{1mm} n \hspace{1mm}\text{is even}, \\ 2^{n-1}- n-1 \qquad\hspace{1mm}  \text{if}\hspace{1mm} n \hspace{1mm} \text{is odd}\end{cases}
\]
is even, and under the hypotheses of the theorem the factor $c_{1}^{m-3}c_3 k^{m}$ does not change its parity. Consequently, Lemma \ref{lem4} follows.

We can now complete the proof of Theorem \ref{th1}, by using first the induction principle on $n$, for $h$ fixed. We
observe that the cases $n=0$ and $n=1$ are trivial, due to the fact that $\tilde{B}_{0}(x)=0$ and $\tilde{B}_1(x)=x$. We assume the property true for indices $m<n$.
Then for the index $n$, we consider the base case $h=0$ and we proceed again by induction on $h \neq 0$.

By virtue of (\ref{2.1}) for $x=\frac{h}{k}$ and $y=\frac{1}{k}$ we get
\begin{equation*}
k^{n}B_{n}^{G}\left( \frac{h+1}{k}\right) =\sum_{m=0}^{n}\binom{n}{m} k^{m}B_{m}^{G}\left( \frac{h}{k}\right) \text{.}
\end{equation*}

For $h=0$, the theorem is obviously true. Assume that it holds for $h\neq 0$. We
obtain:
\[
k^{n}\widetilde{B}_{n}^{G}\left( \frac{h+1}{k}\right)- k^{n}\widetilde{B}_{n}^{G}\left( \frac{h}{k}\right)=\sum_{m=0}^{n-1}\binom{n}{m} k^{m}\tilde{B}_{m}^{G}\left( \frac{h}{k}\right)
+\sum_{m=0}^{n-1}\binom{n}{m} k^{m} {B}_{m}^{G}.  \label{II.32}
\]

The induction hypothesis, together with the base case $h/k=0$, shows that the theorem holds for all rational numbers $h/k$.

\end{proof}

\subsection{Related results}

The next theorem is a direct generalization of the Bartz--Rutkowski theorem, proved in \cite{BR}.

\begin{theorem}
\label{th2} Under the hypotheses of Theorem \ref{th1}, we have that if n is even or n=1,
\begin{equation}
k^{n}B_{n}^{G}\left( \frac{h}{k}\right) +\sum_{\overset{p-1\mid n}{p\nmid k}}\frac{%
c_{p-1}^{n/p-1}}{p}\in \mathbb{Z}\text{;}
\end{equation}
if $n\geq 3$ is odd,
\begin{equation}
k^{n} B_{n}^{G} \left( \frac{h}{k}\right) \in \mathbb{Z}\text{.}
\end{equation}
\end{theorem}

\begin{proof}

If $s$ is a positive integer, then, under the same hypotheses of Theorem \ref{th1}, we have
\begin{equation}
\left( s^{n}-1\right) \sum_{\overset{p-1\mid n}{p \nmid s}}
c_{p-1}^{n/(p-1)}/p\in \mathbb{Z}\text{.}
\label{lemma1}
\end{equation}
This is a consequence of Fermat's little Theorem: if $s\in \mathbb{Z}$, and $p \nmid s$, then $s^{p-1}\equiv 1$ mod $p$. Therefore, if $c_{p-1}\equiv 1$ mod $p$ for all primes $p\geq 2 $, we deduce the congruence (\ref{lemma1}).

In order to obtain the result, it is sufficient to combine Theorem \ref{th1}, the universal von Staudt congruences (\ref{US1}), (\ref{US2}) and relation \eqref{lemma1}.
\end{proof}

Strictly related to the AM theorem is an interesting problem proposed to the author by A. Granville \cite{priv}.

\textbf{Problem}: To classify the sequences of polynomials $\{P_n(x)\}_{n\in \mathbb{N}}$
satisfying the Almkvist--Meurman property
\begin{equation}
k^{n}P_{n}\left( \frac{h}{k}\right) \in \mathbb{Z}\text{, }\forall \text{ }%
h,k,n\in \mathbb{N}, \hspace{3mm} k\neq0  \label{G}.
\end{equation}
We shall call this class of sequences the \textit{Granville class}.

We observe that Theorem \ref{th1} provides a tool for generating a large family of polynomial sequences satisfying the property (\ref{G}).


\section{The Todd genus, the $L$- and $A$-genera and related universal polynomials}

In this section, we will show that there exists a close connection between the $m$--sequences of Hirzebruch's theory of genera and the Lazard formal group, via the universal congruences previously discussed. To this aim, we will introduce two classes of polynomials related to specific cases of the $K$--genus of an almost complex manifold, that can be interpreted as particular instances of the construction of universal Bernoulli polynomials.

To fix the notation, let $R$ be a commutative ring with identity, and $\mathcal{R}=R\left[p_1,p_2,\ldots\right]$ be the ring of polynomials in the indeterminates $p_i$ with coefficients in $R$. We assume that the product $p_{k_1}\ldots p_{k_r}$ has weight $k_1+\ldots k_r$, so that $\mathcal{R}=\sum_{n=0}^{\infty}\mathcal{R}_{n}$, where $\mathcal{R}_{n}$ contains only polynomials of degree $n$. Consider a \textit{multiplicative sequence} (or $m$--sequence) of polynomials $\{K_n\}$ in the indeterminates $p_i$ with $K_{0}=1$ and $K_{n}\in \mathcal{R}_{n}, j\in\mathbb{N}$. As is well known, an $m$--sequence is completely determined by the specification of the associated characteristic power sequence $Q(z)$ (see \cite{Hirz} for details).

The polynomials  $\{T_{k}\left(c_1,\ldots,c_k \right)\}$ of the $m$--sequence associated with the power series $Q(x)=x/\left(1-e^{-x}\right)$ are called the Todd polynomials (here we adopt the conventional notation $x=z^2$ and
\[
\sum_{i=0}^{\infty}p_i(-z)^{i}=\left(\sum_{j=0}^{\infty} c_{j}(-x)^j \right)\left(\sum_{i=0}^{\infty}c_i x^{i}\right)
\]
for the relation among the $p_i$'s and the $c_i$'s).

Now, let $M_{n}$ be a compact, differentiable of class $C^{\infty}$ almost complex manifold of dimension $n$, and let $c_{i}\in H^{2i}(M_n,\mathbb{Z})$ denote the Chern classes of the tangent $GL(n,\mathbb{C})$--bundle of $M_{n}$. The $K$--genus of $M_n$, denoted by $K_{n}[M_n]$, is a ring homomorphism with respect to the cartesian product of two almost complex manifolds.

The genus associated with the sequence of Todd polynomials can be introduced in general for an admissible space $A$ , i.e. a locally compact, finite dimensional space, with $A=\cup_{j\in\mathbb{N}} S_{j}$, where $S_j$ are compact sets. Let us denote by $\mathfrak{b}$ a continuous $GL(q,\mathbb{C})$--bundle over $A$, with Chern classes $c_i\in H^{2i}(A,\mathbb{Z})$. The \textit{total Todd class} is given by $Td(\mathfrak{b})=\sum_{k=0}^{\infty} T_{k}(c_1,\ldots,c_k)$. When $q=1$ and $c_1(\mathfrak{b})=t\in H^2(A,\mathbb{Z})$, we have $Td(\mathfrak{b})=t/\left(1-e^{-t}\right)$. For an almost complex manifold $M_n$, the Todd genus is defined to be the rational number $T_n[M_n]=Td(\mathfrak{b}(M_n))^{(2n)}$, where $u^{(2n)}$ denotes the $2n$--dimensional component of $u\in H^{*}(M_n)$.

A similar construction applies in relation with the theory of $L$--genus, which is the genus associated with the $m$--sequence with characteristic power series $Q(\sqrt{z})=\sqrt{z}/\tanh\sqrt{z}$, and denoted by $L_{k}(p_1,\ldots,p_k)$. The $A$--genus is associated with the $m$--sequence with characteristic power series $Q(z)=2 \sqrt z/\sinh 2\sqrt z$.

The aim of this section is to study the algebro--geometric properties of the $m$--sequences discussed above.
We introduce now two families of polynomials related to the genus of these sequences.
\begin{definition}
We shall call $\alpha$--polynomials the family of polynomials generated by the characteristic power series associated with the A--genus:
\begin{equation}
\frac{2 t e^{yt}}{\sinh{2 t}}=\sum_{k\geq 0}\alpha_{k}\left( y\right) \frac{t^{k}}{k!}.
\end{equation}
\end{definition}
The  $\alpha$--polynomials are particular cases of the Universal Bernoulli polynomials \eqref{I.3}; also the Todd polynomials are directly related to the N\"orlund Bernoulli polynomials.
The next result relates these polynomials with the universal AM congruences.
\begin{proposition} \label{Lemma1}
 Let $\widetilde{\alpha}_{n}(x)$ denote the polynomials generated by
 \begin{equation}
\frac{2 t (e^{yt}-1)}{\sinh{2 t}}=\sum_{k\geq 0}\widetilde{\alpha}_{k}\left( y\right) \frac{t^{k}}{k!}.
\end{equation}
The sequence $\{\widetilde{\alpha}_{n}(x)\}_{n\in\mathbb{N}}$ is in the Granville class.
\end{proposition}
\begin{proof}
It is sufficient to observe that the following identity holds:
\beq
\frac{t}{2 \sinh t}=t/(e^t-1)-t/(e^{2t}-1).
\eeq
Consequently, the sequence $\{\widetilde{\alpha}_{n}(x)\}_{n\in\mathbb{N}}$ satisfies the hypotheses of Theorem \ref{th1}. The details are left to the reader.
\end{proof}
The Cauchy formula relates $Q(z)$ with $s_n$, which is defined to be the coefficient of $p_n$ in the polynomial $K_n$:
\beq
1-z\frac{d}{dz} \log Q(z)=\sum_{j=0}^{\infty}(-1)^{j}s_j z^{j}.
\eeq
A similar construction can be proposed in relation with the $L$--genus.
\begin{definition}
We shall define the $\lambda$--polynomials to be the family of polynomials generated by the characteristic power series associated with the L--genus:
\begin{equation}
\frac{t e^{yt}}{\tanh{ t}}=\sum_{k\geq 0}\lambda_{k}\left( y\right) \frac{t^{k}}{k!}.
\end{equation}
\end{definition}
\noindent A completely analogous result also holds for the associated polynomials $\{\widetilde{\lambda}_{n}(x)\}_{n\in\mathbb{N}}$.
\begin{proposition} \label{Lemma2}
 Let $\widetilde{\lambda}_{n}(x)$ denote the polynomials generated by
 \begin{equation}
\frac{t (e^{yt}-1)}{\tanh{t}}=\sum_{k\geq 0}\widetilde{\lambda}_{k}\left( y\right) \frac{t^{k}}{k!}.
\end{equation}
The sequence $\{\widetilde{\lambda}_{n}(x)\}_{n\in\mathbb{N}}$ is in the Granville class.
\end{proposition}
\section{Universal congruences and integer sequences}

\subsection{Sequences of integer numbers}
In this section, both sequences of integer numbers and Appell sequences of
polynomials with integer coefficients are constructed as a byproduct of
the previous theory.

\begin{lemma}
\label{lem5} Consider a sequence of the form
\begin{equation}
\frac{t}{G_{1}\left( t\right) }-\frac{t}{G_{2}\left( t\right) }%
=\sum_{k=0}^{\infty }\frac{\,N_{k}}{2 \hspace{1mm} k!}t^{k}\text{,}  \label{7.1}
\end{equation}%
where $G_{1}\left( t\right) $ and $G_{2}\left( t\right) $ are formal group
exponentials, defined as in formula (\ref{I.2}). Assume that $c_i^{G_{j}}\in\mathbb{Z}$ for all $i\in\mathbb{N}$, $j=1,2$ and $c_{p-1}^{G_{1}}\equiv c_{p-1}^{G_{2}}$ $mod$ $p$ for all $p\geq 2$ (here $c_{n}^{G_{j}}$ denotes the $n^{th}$ coefficient of the expansion \eqref{I.1} for the logarithm associated to $G_j$). Then $\{N_{k}\}_{k\in \mathbb{N}}$ is a sequence of integers.
\end{lemma}

\begin{proof} The Bernoulli--type numbers $B_{k}^{G_1}$ and $B_{k}^{G_2}$ associated with the formal
group exponentials $G_{1}(t)$ and $G_{2}(t)$, under the previous assumptions for $k$ even
must satisfy Clarke's universal congruence (\ref{US1}). For $k$ odd, these
numbers are integers or half--integers, due to (\ref{US2}). It follows that the
difference $B_{k}^{G_1}-B_{k}^{G_2}$ for  $k$ even is an
integer and for $k$ odd is a half--integer or an integer. The result follows from the definition of $\{N_{k}\}_{k\in \mathbb{N}}$ in \eqref{7.1}.
\end{proof}
\subsection{Construction of integer sequences}

a) The characteristic power series obtained as the difference between the power series associated with the $L$--genus and that one associated with the Todd genus, i.e. $Q(t)=t/\tanh t- t/(1-e^{-t})$ is the generating function of the sequence
\beq
-1,1,0,-1,0,3,0,-17,0,155,\ldots
\eeq

b) In \cite{Tempesta3}, realizations of the universal Bernoulli polynomials were constructed by using the finite operator theory
\cite{Rota}. Here we quote two generating functions of these classes of polynomials, related to certain difference delta operators:
\beqa
\nn \sum_{k=0}^{\infty }\frac{B_{k}^{V}\left( x\right) }{k!}t^{k}&=&\frac{te^{xt}}{%
e^{3t}-2e^{2t}+2e^{t}-2e^{-t}+e^{-2t}}\text{,}  \\
\sum_{k=0}^{\infty }\frac{B_{k}^{VII}\left( x\right) }{k!}t^{k}&=&\frac{%
-te^{xt}}{e^{4t}-e^{3t}+e^{2t}-2e^{t}+e^{-t}-e^{-2t}+e^{-3t}}\text{.}
\label{BVII}
\eeqa

The sequences associated with the generating functions (\ref{BVII}) satisfy the classical
Clausen--von Staudt congruence. Since $N_{k}=2(B_{k}^{G_{1}}-B_{k}^{G_2})$, they can be used to
construct integer sequences of numbers and polynomials. Here are reported
some representatives of generating functions of integer sequences which can
be obtained from the previous considerations.

b.1)
\begin{equation*}
\frac{t(1+e^{t})}{2(1+e^{t}-e^{2t})}\text{.}
\end{equation*}%
The sequence generated is $\hspace{3mm} 2,6,\text{ }39,\text{ }324,\text{ }3365,\text{ }41958,\ldots$

b.2)
\begin{equation*}
\frac{-2t\left( 1+2\cosh t+4\cosh 2t-6\sinh t\right) }{(-6+8\cosh t)(2+\cosh
t-\cosh 2t-\sinh t+\sinh 2t+2\sinh 3t)}\text{.}
\end{equation*}%
The sequence generated is%
$
\hspace{3mm} -7,\text{ }61,-642,\text{ }10127,-207110,\text{ }5001663,...
$

\subsection{Polynomial sequences with integer coefficients}

As an immediate consequence of the previous results, we can also construct new sequences of Appell polynomials possessing \textit{integer coefficients}.

\begin{lemma}
Assume the hypotheses of Lemma \ref{lem5}. Then any sequence of
polynomials of the form
\begin{equation}
\left[\frac{t}{G_{1}\left( t\right)} - \frac{t}{G_{2}\left( t\right)}\right]
e^{xt}=\sum_{k=0}^{\infty }\frac{\,N_{k}\left( x\right) }{2 k!}t^{k}
\label{7.3}
\end{equation}
is an Appell sequence with integer coefficients.
\begin{proof}
The generating function \eqref{7.3}, being of the form $f(t)e^{xt}$, determines an Appell sequence of polynomials. In particular, the coefficients of each polynomial $N_{k}\left( x\right)$ of the sequence are determined by the function $f(t)$. The thesis follows by combining the result of the previous Lemma \ref{lem5} with the choice $f(t)=\left[\frac{t}{G_{1}\left( t\right)} - \frac{t}{G_{2}\left( t\right)}\right]$.
\end{proof}
\end{lemma}
\noindent As an example, consider the sequence of polynomials $\left\{p_{n}(x)\right\}_{n\in\mathbb{N}}$, generated by
\begin{equation}
-\frac{te^{-\frac{3t}{2}}\sec h(\frac{t}{2})\left[
4+e^{t}(1+2e^{t}(-1+e^{t}))\right] }{2(-3+4\cosh t)(1+(-2+4\cosh t)\sinh t)}
e^{xt}\text{.}\label{genf1}
\end{equation}
It is easy to verify that (\ref{genf1}) is the generating function of a sequence of Appell polynomials. The first polynomials of the sequence are
\begin{equation*}
p_{0}(x)=-5,\qquad p_{1}(x)=29-10x,
\end{equation*}
\begin{equation*}
p_{2}(x)= -150+87x-15x^{2},
\end{equation*}
\begin{equation*}
p_{3}(x)=1279-600x+174x^{2}-20x^{3}
\end{equation*}
\begin{equation*}
p_{4}(x)=-17770+6395x-1500x^{2}+290x^{3}-25x^{4},\ldots .
\end{equation*}

\section{A class of Riemann--Hurwitz zeta functions and their values at negative integers}

The previous results enable us to compute special values of a new class of zeta functions, recently introduced in \cite{Tempesta3} (here we propose a different formulation of this class). An important particular example of this class is the celebrated Riemann--Hurwitz zeta function. In order to make this Section self--consistent, some necessary definitions are recalled.

\subsection{Hurwitz zeta functions and formal group laws}

\begin{definition} \label{Hur}
\textit{Let $G(t)$ be a formal group exponential, such that $e^{-at}/G(t)$ is a $%
C^{\infty }\,$function over $\mathbb{R}_{+}$, rapidly decreasing at
infinity. The generalized Hurwitz zeta function associated with G is the
function $\zeta ^{G}\left(s,a\right) $, defined for \textit{Re}$(s)>1$ and $a>0$ by}
\begin{equation}
\zeta ^{G}\left( s,a\right) =\frac{1}{\Gamma \left( s\right) }%
\int_{0}^{\infty}\frac{e^{-a x}}{G\left( x\right) }x^{s-1}dx\label{RH}
\text{.}
\end{equation}
\end{definition}

 When $G(x)= 1-e^{-x}$,  we obtain the classical Hurwitz zeta function, which converges absolutely for \textit{Re}(s) $>1$ and extends to a meromorphic function in $\mathbb{C}$, represented as a Mellin transform
\begin{equation*}
\zeta \left( s,a\right) =\sum_{n=0}^{\infty }\left( n+a\right) ^{-s}=\frac{1}{\Gamma \left( s\right) }\int_{0}^{\infty }%
\frac{e^{-ax}}{1-e^{-x}}x^{s-1}dx.
\end{equation*}%

\noindent By means of an analysis of the singularities of both the integral in \eqref{RH} and of $\Gamma(s)$, we deduce that $\zeta^{G}(s,a)$ can be analytically continued to the whole complex plane, with a simple pole at $s=1$. Consequently, we get the following result, generalizing a classical property of the standard Bernoulli polynomials.

\begin{proposition}\label{corollary}
For any $m\in\mathbb{N}$, the following property holds:
\begin{equation}
\zeta ^{G}\left( -m, a\right) =-\frac{B_{m+1}^{{G'}}\left( a\right) }{m+1}\text{,}\label{zeta}
\end{equation}
where according to eq. \eqref{I.3}, $B_{m}^{{G'}}\left( x\right) $ is the m--th generalized Bernoulli
polynomial associated with the formal group exponential ${G}'(t):=-G(-t)$ (with $\alpha=1$).
\end{proposition}

The main result of this section is the following interesting congruence for the generalized Riemann--Hurwitz zeta functions (\ref{RH}).

\begin{theorem}
\label{th3} Let $G\left( t\right)$ be a formal group exponential, according to Def. \ref{Hur}, such that $c_i \in \mathbb{Z}$ for all $i=1,2,\ldots$. Assume that $c_{p-1}\equiv 1$ mod $p$ for all primes $p\geq
2 $, and $c_1\equiv c_3$ mod $2$, or $c_1$ odd and $c_3$ even. Then

\begin{equation}
nk^{n} \zeta^{G}(1-n,h/k)\in \mathbb{Z}\text{,}  \label{RHc1}
\end{equation}

for all $h>0, k>0$ integers and $n$ odd, $n\geq3$.
%
%
\end{theorem}

\begin{proof}
We can use the previous Proposition \ref{corollary} with $m=n-1$ in the property \eqref{zeta}. The result is a consequence of the statement of Theorem \ref{th2}.
\end{proof}

The case  $n$ even or $n=1$ is very similar and is left to the reader.
\subsection{The $\chi$--universal numbers from Dirichlet characters and congruences}

Let us consider now the case of zeta functions, associated with formal group laws, depending on a Dirichlet character. In the following, given $N\in\mathbb{N}$, $N\neq 0$, we will denote by $\chi(n)$ a \textit{Dirichlet character with conductor $N$}.

We propose the following definition.

\begin{definition}
Let $\chi $ be a nontrivial Dirichlet character of conductor \thinspace $N$.
The Bernoulli $\chi $--numbers $B_{n,\chi }^{G}\,$ related with
the formal group law associated with the formal group exponential $G$ are defined by
\begin{equation}
B_{n,\chi }^{G}:=N^{n-1}\sum_{a=1}^{N}\chi \left( a\right) B_{n}^{G}\left(
\frac{a}{N}\right) . \label{chinumb}
\end{equation}
Similarly, the universal Bernoulli $\chi $--numbers are defined by
\begin{equation}
\widehat{B}_{n,\chi }:=N^{n-1}\sum_{a=1}^{N}\chi \left( a\right) \widehat{B}_{n}\left(
\frac{a}{N}\right) . \label{chinumb}
\end{equation}
\end{definition}
It is possible to define zeta functions related to Dirichlet characters and formal group laws via the formula
\begin{equation}
L\left( G,\chi ,s\right) =\frac{1}{N^{s}}\sum_{j=1}^{N }\chi \left( j\right)
\zeta ^{G}\left( s,\frac{j}{N}\right) \text{.} \label{L}
\end{equation}
The following result provides a congruence for the class of zeta functions (\ref{L}).
\begin{corollary}
\label{th4} Under the hypotheses of Theorem \ref{th3}, we have for $n$ odd, $n\geq3$
\begin{equation}
n N L\left( G,\chi, 1-n\right)\in \mathbb{Z (\chi})\text{,}  \label{RHc}
\end{equation}

where $\mathbb{Z (\chi})$ denotes the integral extension of $\mathbb{Z}$ generated by the values of the non-trivial character $\chi$.
\end{corollary}

\begin{proof}
It is sufficient to invoke the same argument used in Theorem \ref{th3}.
\end{proof}

\section*{Appendix}
Here we report the explicit expressions of the first universal Bernoulli polynomials.
\beq
\widehat{B}_0(x)=1 \quad \widehat{B}_1(x)=x + c_1 /2, \quad \widehat{B}_2(x)=x^2 + x c_1 - c_{2}/2 + 2 c_2 /3, \nn
\eeq
\beq
\widehat{B}_3(x)=x^3 +
 3/2 \hspace{1mm} x^2 c_1 + \left(2 c_2- 3/2  \hspace{1mm} c_{1}^2\right) x + (3 c_{1}^3)/2  -
 3 c_1 c_2 + 3 c_3/2, \nn
\eeq
\beqa
\widehat{B}_4(x)&=&x^4 + 2 x^3 c_1 + \left(4 c_2 - 3 c_{1}^2\right)x^2
+ \left(6 c_{1}^3  - 12 c_1 c_2 + 6 c_3\right)x  \nn \\&+&  20 c_{1}^2 c_2 - 16/3  \hspace{1mm} c_{2}^2  - 12 c_1 c_3 +
 24/5  \hspace{1mm} c_4 - 15 c_{1}^4/2, \nn
\eeqa
\beqa
\widehat{B}_5(x)&=&x^5 + 5/2  \hspace{1mm} x^4 c_1 + \left(20/3  \hspace{1mm} c_{2}- 5 c_{1}^2\right) x^3 +
\left(15  c_{1}^3 - 30 c_{1} c_{2} + 15 c_3\right) x^2  \nn \\
\noindent\nn &+& \left(100  \hspace{1mm} c_{1}^2 c_{2} -
 75/2 \hspace{1mm}  c_{1}^4 - 80/3  \hspace{1mm} c_{2}^2 - 60 c_{1} c_{3} + 24 c_{4}\right) x +
 105/2  \hspace{1mm} c_{1}^5  - 175 c_{1}^3 c_{2}  \\
\noindent\nn &+& 100  \hspace{1mm} c_{1} c_{2}^2 + 225/2  \hspace{1mm} c_{1}^2 c_3 -
 50 c_{2} c_{3}  - 60 c_{1} c_{4} + 20 c_{5}.
\eeqa

\textbf{Acknowledgments}

I am indebted to the unknown referee for many useful observations, which improved the quality of the paper.

I wish to thank heartily Prof. D. Zagier for interest in the work and for correcting a wrong statement in an early version of the paper. I am grateful to Prof. A. Granville for suggesting the cited Problem. I am also indebted with Prof. R. A. Leo for many interesting discussions and helpful suggestions.

Part of this research work has been carried out in the Centro di ricerche matematiche Ennio De Giorgi, Scuola Normale Superiore, Pisa, that I thank for kind hospitality.

The support from the research project FIS2011--22566, Ministerio de Ciencia e Innovaci\'{o}n, Spain is gratefully acknowledged.

\end{document}